\documentclass[reqno,11pt]{amsart}

\usepackage{amsmath}
\usepackage{amssymb}

\newcommand{\D}{\mathrm{D}}
\newcommand{\E}{\mathrm{E}}
\newcommand{\runinhead}{\vspace{3mm} \noindent\textbf}

\newtheorem{theorem}{Theorem}[section]
\newtheorem{lemma}[theorem]{Lemma}
\newtheorem{proposition}[theorem]{Proposition}
\theoremstyle{definition}
\newtheorem{ass}[theorem]{Assumption}

\DeclareMathOperator{\supp} {supp}
\DeclareMathOperator{\esssup} {ess\, sup}

\newcommand{\1}{{\mathchoice {1\mskip-4mu\mathrm l}
{1\mskip-4mu\mathrm l}
{1\mskip-4.5mu\mathrm l} {1\mskip-5mu\mathrm l}}}

\newcommand{\zd}{\mathbb{Z}^d}
\newcommand{\nn}{\mathbb{N}}
\newcommand{\rd}{\mathbb{R}^d}
\newcommand{\rr}{\mathbb{R}}
\newcommand{\ee}{\mathbb{E}}
\newcommand{\pp}{\mathbb{P}}
\newcommand{\kh}{K_H}
\newcommand{\masz}{\mathcal{M}_1}
\newcommand{\chid}{\chi^{\mathrm{d}}}
\newcommand{\chic}{\chi^{\mathrm{c}}}
\newcommand{\erwxi}[1]{\langle{#1}\rangle}
\newcommand{\Erwxi}[1]{\big\langle{#1}\big\rangle}
\newcommand{\norm}[1]{\Arrowvert{#1}\Arrowvert}
\newcommand{\floor}[1]{\lfloor {#1}\rfloor}
\newcommand{\ssup}[1] {{{\scriptscriptstyle{({#1}})}}}
\newcommand{\scal}[2]{\big(#1,#2\big)}
\newcommand{\scaltext}[2]{(#1,#2)}

\renewcommand{\subset}{\subseteq}

\begin{document}

\title[The PAM with Acceleration and Deceleration]{The Parabolic Anderson Model\smallskip\\ with Acceleration and Deceleration}

\author[Wolfgang K\"onig and Sylvia Schmidt]{}

\maketitle

\thispagestyle{empty}
\vspace{0.2cm}

\centerline{By {\sc Wolfgang K\"onig}\footnote{Weierstra\ss-Institut Berlin, Mohrenstr.~39, 10117 Berlin, and Institut f\"ur Mathematik, Technische Universit\"at Berlin, Str.~des 17.~Juni 136, 10623 Berlin, Germany, {\tt koenig@wias-berlin.de}} and {\sc Sylvia Schmidt}\footnote{Mathematisches Institut, Universit\"at Leipzig, Augustusplatz 10/11, 04009 Leipzig, Germany, {\tt sylvia.schmidt@math.uni-leipzig.de}} }

\vspace{1cm}

\begin{quote}
{\it Abstract.} We describe the large-time moment asymptotics for the parabolic Anderson model where the speed of the diffusion is coupled with time, inducing an acceleration or deceleration. We find a lower critical scale, below which the mass flow gets stuck. On this scale, a new interesting variational problem arises in the description of the asymptotics. Furthermore, we find an upper critical scale above which the potential enters the asymptotics only via some average, but not via its extreme values. We make out altogether five phases, three of which can be described by results that are qualitatively similar to those from the constant-speed parabolic Anderson model in earlier work by various authors. Our proofs consist of adaptations and refinements of their methods, as well as a variational convergence method borrowed from finite elements theory.
\end{quote}

\vfill

\noindent{\it MSC 2000.} 35K15, 82B44, 60F10, 60K37.\\

\noindent{\it Keywords and phrases.} Parabolic Anderson model, moment asymptotics, variational formulas, accelerated and decelerated diffusion, large deviations, random walk in random scenery.

\numberwithin{equation}{section}

\eject

\section{Introduction}

We consider the solution $u^{\ssup t}\colon [0,\infty)\times\zd\to[0,\infty)$, $t>0$, to the Cauchy problem for the heat equation with random coefficients and $t$-dependent diffusion rate,
\begin{align}
\label{pam}
\frac{\partial}{\partial s} u^{\ssup t}(s,z) &=\kappa(t) \Delta u^{\ssup t}(s,z)+\xi(z)u^{\ssup t}(s,z), \quad s>0,\, z\in\zd, \\
u^{\ssup t}(0,\cdot) &= \1_0, \nonumber
\end{align}
where $\Delta$ is the discrete Laplacian,
$$
\Delta f(z) = \sum_{x\in\zd\colon |x-z|=1}[f(x)-f(z)],
$$
$(\xi(z))_{z\in\zd}$ is a field of independent and identically distributed random variables, and $\kappa\colon[0,\infty)\to[0,\infty)$ is a~function with $\lim_{t\to\infty}t\kappa(t)=\infty$. Our main goal is to understand the asymptotic behaviour as $t\to\infty$ of the expected total mass at time $t$,
\begin{equation*}
U(t) = \sum_{z\in\zd} u^{\ssup t}(t,z).
\end{equation*}
The total mass may be represented in terms of the famous {\it Feynman--Kac formula},
\begin{equation}\label{FKform}
U(t)=\ee_0^{\ssup t}\Big[\exp\Big\{\int_0^t\xi(X_s)\,\D s\Big\}\Big],
\end{equation}
where $(X_s)_{s\in[0,\infty)}$ is a random walk with generator $2d\kappa(t)\Delta$, starting from zero under $\ee_0^{\ssup t}$. Denoting by $\erwxi{\,\cdot\,}$ the expectation with respect to the random potential~$\xi$, we will study the logarithmic asymptotics of~$\erwxi{U(t)}$ for various choices of the diffusion function $t\mapsto \kappa(t)$.

The model with constant diffusion rate $\kappa(t)\equiv 1$ has been analysed in \cite{GM98} and~\cite{BK01} for three important classes of tail distributions of~$\xi(0)$, see also \cite{GK05} for a survey and \cite{CM94} for more background. In \cite{HKM06} a~classification of all potential distributions into four universality classes was made out such that the qualitative behaviour of~$\erwxi{U(t)}$ in each of the classes is similar. This classification holds under mild regularity assumptions and depends only on the upper tails of the potential. Heuristically, the main effect in each of these classes is the concentration of the total mass on a so-called intermittent island the size of which is $t$-dependent and deterministic. The (rescaled) shape of the solution and the potential on this island can be described by a deterministic variational formula. The thinner the tails of the potential distribution are, the larger the islands are, ranging from single sites to large areas, however still having a radius $\ll t^{1/d}$.

In~\eqref{pam}, the diffusion is coupled with time so that it is accelerated if the diffusion function $t\mapsto\kappa(t)$ grows or decelerated if it decreases. Now an interesting competition between the speed of the diffusion and the thickness of the tails of the potential distribution arises: the faster $\kappa(t)$ is, the stronger the flattening effect of the diffusion term is. One rightfully expects that if the speed of this function is not too extreme, then similar formulas should be valid as for constant diffusion rate. Indeed, we will identify a lower critical scale for $\kappa(t)$, which depends on the upper tails of the potential distribution, and marks the threshold below which the mass does not flow unboundedly far away from the origin in the Feynman--Kac formula, see below Assumption~\ref{ass_h}. Then we are in the case of \cite{GM98}.
Furthermore, we will see that -- if $\kappa(t)$ is above this lower critical scale -- $t^{2/d}$ presents an upper critical scale in the sense that, for $\kappa(t)\ll t^{2/d}$, the main contribution to the total mass comes from extremely high potential values, while for $\kappa(t)\approx t^{2/d}$, it comes from just super-average, but not extreme, values. This is reflected by the fact that the asymptotics can be described in terms of the upper tails of the potential distribution in the former case (then we find the formulas derived in \cite{BK01} and \cite{HKM06}), but all the details of this distribution are required in the latter.  (If the speed is even faster, then, conjecturally, only a rough mean behaviour of the potential values will influence the asymptotics.)

The paper is organised as follows. In Section~\ref{sec_assumptions}, we formulate our assumptions on the potential and on the function~$\kappa$. Then we state our results for the moment asymptotics of~$U(t)$ in Section~\ref{sec_results}. Our main result will be the identification of five phases with qualitatively different behaviour, which we will describe informally in Section~\ref{sec_phases} and rigourously in Section~\ref{sec_asymptotik} (for four of them). We will also give a proposition concerning the convergence of a~discrete variational formula to the corresponding continuous version, representing one of the main tools used in the proof of the asymptotics. In Sections~\ref{sec_proofs}--\ref{sec-Proof4}, we give sketches of the proofs of this proposition and of the theorems. The details are rather lengthy and involved; they may be found in the second author's thesis \cite{S10}.

\section{Assumptions and Preliminaries} \label{sec_assumptions}

\subsection{Model Assumptions} \label{sec_ass}
Let
$$
H(t)=\log\erwxi{\E^{t\xi(0)}},\qquad t>0,
$$
be the logarithmic moment generating function of~$\xi(0)$. We assume $H(t)<\infty$ for all~$t>0$, which is sufficient for the existence of a~nonnegative solution of~\eqref{pam} and the finiteness of all its positive moments \cite{GM90}. Now we recall the discussion on regularity assumptions in \cite[Section~1.2]{HKM06}. If we assume that $t\mapsto H(t)/t$ is in the de Haan class, then the theory of regularly varying functions provides us with an asymptotic description of~$H$ that depends only on two parameters~$\gamma$ and~$\rho$, see \cite{BGT87} and \cite[Proposition~1.1]{HKM06}. This leads to the following assumption which will be in force throughout the rest of this paper.
\begin{ass}
\label{ass_h}
There exist parameters~$\gamma\ge 0$ and~$\rho>0$ and a~continuous function $\kh\colon(0,\infty)\to(0,\infty)$, regularly varying with parameter~$\gamma$, such that, locally uniformly in $y\in[0,\infty)$,
\begin{equation}
\label{konvergenz_h}
\lim_{t\to\infty}\frac{H(ty)-yH(t)}{\kh(t)} = \rho\widehat{H}(y),
\end{equation}
where
\begin{equation}
\label{hat_h}
\widehat{H}(y)=\begin{cases}
    y \log y & \textrm{if } \gamma=1, \\[2mm]
    \dfrac{y-y^\gamma}{1-\gamma} & \textrm{if } \gamma\ne 1.
            \end{cases}
\end{equation}
\end{ass}
The scale function $\kh$ roughly describes the thickness of the potential tails at infinity. As we will see later, the function $t\mapsto\kh(t)/t$ presents a lower critical scale for the diffusion function $\kappa(t)$. The following lemma is a consequence of \cite[Theorem~3.6.6]{BGT87}.
\begin{lemma}
\label{regvar_h}
Let Assumption~\textup{\ref{ass_h}} hold.
\begin{enumerate}
\item[\textup{(a)}] If $\esssup \xi(0) \in\{0,\infty\}$, then~$H$ is regularly varying with index~$\gamma$.
\item[\textup{(b)}] If $\erwxi{\xi(0)}=0$, then~$H$ is regularly varying with index~$\gamma\vee 1$.
\end{enumerate}
\end{lemma}

Now we formulate some mild regularity assumptions on the speed function~$\kappa$.
\begin{ass}
\label{ass_kappa}The following limits exist:
$$
\lim_{t\to\infty} t\kappa(t)=\infty,\qquad
\lim_{t\to\infty}\frac{t\kappa(t)}{\kh(t)}\in[0,\infty],
\qquad \lim_{t\to\infty}\frac{\kappa(t)}{t^{2/d}}\in[0,\infty].
$$
\end{ass}

We also need a scale function $\alpha\colon[0,\infty)\to[0,\infty)$, which will be interpreted as the order of the radius of the relevant island. While we can define $\alpha=1$ in the results for Phases~1 and~2 of our classification, we will need the following fixed point equation in Phase~3:
\begin{equation}
\label{alpha_def}
\kh\Big(\frac{t}{\alpha_t^d}\Big)=\frac{t\kappa(t)}{\alpha_t^{d+2}}.
\end{equation}
Let us state existence and some important properties of a solution of~\eqref{alpha_def}.
\begin{lemma}
\label{alpha_asymp}
Let $\kappa(t)$ be regularly varying with index $\beta\in(\gamma-1,2/d)$. Then there exists a regularly varying function $\alpha$ such that~\eqref{alpha_def} holds for all large~$t$. Any solution $\alpha(t)=\alpha_t$
satisfies $\lim_{t\to\infty}\alpha_t=\infty$.
Furthermore, $t/\alpha_t^d\gg 1$ and $\alpha_t^{x}\ll t\kappa(t)$ for each $x<d+2$.
\end{lemma}
\begin{proof}
Similar to the proof of~\cite[Proposition~1.2]{HKM06}. For details, see \cite[Lemma~2.1.5]{S10}.
\end{proof}

From the assumptions of Theorem~\ref{theo_phase1bis3}(c) below, we will see that the interval for the index of regular variation for~$\kappa$ is not a~hard restriction in Phase~3.

\subsection{Variational Formulas} \label{sec_vario}

The following variational formulas will play a role in our results. Here, $\mathrm{H}^1(\rd)$ is the Sobolev space on~$\rd$ and $\masz(\zd)$ is the space of probability measures on~$\zd$. The inner product on~$\zd$ is denoted by~$\scaltext{\cdot\,}{\cdot}$. All integrals are with respect to Lebesgue measure. We always have $\rho,\theta>0$ and $\gamma\ge 0$.
\begin{align}
\label{chi_b}
\chi^{\ssup {\rm B}}_\gamma(\rho)&=\inf_{\substack{g\in \mathrm{H}^1(\rd)\\ \norm{g}_2=1}}
        \Big\{\int_{\rd} |\nabla g|^2 + \frac{\rho}{1-\gamma}\int_{\rd} (g^{2\gamma}-g^2) \Big\},  \\
\label{chi_fb}
\chi^{\ssup {\rm {AB}}}(\rho)&=\inf_{\substack{g\in \mathrm{H}^1(\rd)\\ \norm{g}_2=1}}
        \Big\{\int_{\rd} |\nabla g|^2 - \rho\int_{\rd} g^2 \log g^2\Big\},  \\
\label{chi_de}
\chi^{\ssup{\rm{DE}}}(\rho)&=\inf_{p\in\masz(\zd)}
        \Big\{{-}\scal{\Delta\sqrt{p}}{\sqrt{p}} - \rho\scal{p}{\log p}\Big\},\\
\label{chi_db}
\chi^{\ssup{\rm {DB}}}_\gamma(\rho)&=\inf_{p\in\masz(\zd)}
        \Big\{{-}\scal{\Delta\sqrt{p}}{\sqrt{p}} + \frac{\rho}{1-\gamma}\scal{p^{\gamma}-p}{1}\Big\},\\
\label{chi_rwrs}
\chi^{\ssup{\rm{RWRS}}}_H(\theta) &= \inf_{\substack{g\in \mathrm{H}^1(\rd)\\ \norm{g}_2=1}}
        \Big\{\int_{\rd} |\nabla g|^2-\theta\int_{\rd}H\circ g^2\Big\}.
\end{align}
If $\gamma=0$, then we use the interpretation $\int_{\rd}g^{2\gamma}=|{\supp g}|$ and $\scaltext{p^{\gamma}}{1}=|{\supp p}|$. We sometimes refer to the formulas that are defined in $\rd$ (that is, $\chi_\gamma^{\ssup{\rm B}}$, $\chi^{\ssup{\rm {AB}}}$ and $\chi_H^{\ssup{\rm {RWRS}}}$) as to \lq continuous\rq\ formulas and to the others as to the \lq discrete\rq\ ones. Clearly, $\chi_\gamma^{\ssup{\rm B}}$ and $\chi^{\ssup{\rm {AB}}}$ are the continuous variants of $\chi_\gamma^{\ssup{\rm{DB}}}$ and $\chi^{\ssup{\rm{DE}}}$, respectively. Note that $\chi^{\ssup{\rm B}}_\gamma$ is degenerate in the case $\gamma>1+2/d$ (which we do not consider here).

The formulas $\chi^{\ssup{\rm{DE}}}$, $\chi^{\ssup{\rm{AB}}}$ and $\chi^{\ssup{\rm{B}}}_\gamma$ are already known from the study of the parabolic Anderson model for constant diffusion $\kappa(t)\equiv 1$ in three universality classes, see the summary in \cite{HKM06}. Our notation refers to the names of these classes introduced there: \lq DE\rq\ for \lq double-exponential\rq, \lq AB\rq\ for \lq almost bounded\rq, and \lq B\rq\ for \lq bounded\rq. Informally, the functions $g^2$ and~$p$, respectively, in the formulas have the interpretation of the shape (up to possible rescaling and vertical shifting) of those realisations of the solution $u^{\ssup t}(t,\cdot)$ that give the overwhelming contribution to the expected total mass, $\erwxi{U(t)}$. If the total mass comes from an unboundedly growing island, then a~rescaling is necessary, and a continuous formula arises, otherwise a discrete one.

In \cite{S09} the existence, uniqueness (up to shift) and some characterisations of the minimiser of $\chi^{\ssup {\rm B}}_\gamma$ are shown for~$\gamma<1$, in \cite{HKM06} it is shown that the only minimiser of $\chi^{\ssup {\rm {AB}}}$ is an explicit Gaussian function, and in \cite{GM98} and \cite{GH99}, the minimisers of $\chi^{\ssup{\rm{DE}}}(\rho)$ are analysed, which are unique (up to  shifts) for any sufficiently large $\rho$. Formula $\chi_H^{\ssup{\rm{RWRS}}}$ is a rescaling of the Legendre transform of a variational formula which appeared in the study of large deviations for the random walk in random scenery in \cite{GKS07}, see~\eqref{gks07_rwrs}. Its properties have not been analysed yet.

However, formula $\chi^{\ssup{\rm{DB}}}_\gamma$ (\lq DB\rq\ refers to \lq discrete bounded\rq) appears in the study of the parabolic Anderson model for the first time in the present paper. Here are some of its properties.
\begin{proposition}
\label{vario}
\begin{enumerate}
\item[\textup{(a)}]
For any~$\rho>0$ and any $\gamma\ne 1$ with $0\le\gamma<\max\{1+1/d,1+\rho/(2d)\}$, there exists a~minimiser for $\chi^{\ssup{\rm {DB}}}_\gamma(\rho)$.
\item[\textup{(b)}]
Let $p$ be a minimiser for $\chi^{\ssup{\rm {DB}}}_\gamma(\rho)$. Then $\supp p$ is finite if and only if $\gamma\le 1/2$. In the case~$\gamma>1/2$ the support of~$p$ is the whole lattice.
\end{enumerate}
\end{proposition}

\begin{proof}
See \cite[Prop.~2.1.8]{S10}. This uses ideas from~\cite[Lemma~3.2]{GK09} for the existence and from \cite[p.~44]{GH99} for the size of the support.
\end{proof}

Similarly to the continuous analogue in \cite[Proposition~1.16]{HKM06}, it is possible to show that
$\lim_{\gamma\to 1}\chi^{\ssup{\rm {DB}}}_\gamma(\rho)=\chi^{\ssup{\rm {DE}}}(\rho)$, furthermore we have $\lim_{\rho\to\infty}\chi^{\ssup{\rm {DB}}}_\gamma(\rho)=2d$.

\section{Results} \label{sec_results}

In what follows, we will use the notation $f_t\gg g_t$ if $\lim_{t\to\infty}f_t/g_t=\infty$ and $f_t\asymp g_t$ if $\lim_{t\to\infty}f_t/g_t$ exists in $(0,\infty)$. We will always work under the assumptions made in Section~\ref{sec_ass}.

\subsection{Five Phases} \label{sec_phases}

Depending on the ratio between the speed~$\kappa(t)$ and the critical scales $\kh(t)/t$ and~$t^{2/d}$, we make out up to five phases. In the following, we resume heuristically our results for these phases. Recall the Feynman--Kac formula in~\eqref{FKform}.

\runinhead{Phase 1.}  $\kappa(t)\ll \kh(t)/t$.\\
The mass stays in the origin, where the potential takes on its highest value. The expected total mass behaves therefore like $\erwxi{U(t)}\approx\erwxi{u^{\ssup{t}}(t,0)}\approx\exp(H(t)-2d t\kappa(t))$. This includes the single-peak case of \cite{GM98}.

\runinhead{Phase 2.} $\kappa(t) \asymp \kh(t)/t$.\\
The radius of the intermittent island remains bounded in time, and consequently the moment asymptotics are given in terms of a discrete variational formula. Denoting $\kappa_*=\lim_{t\to\infty}t\kappa(t)/\kh(t)$,
\begin{equation}
\label{result}
\lim_{t\to\infty}\frac{1}{t\kappa(t)}\log \erwxi{U(t)\E^{-H(t)}}
=-\begin{cases}
\chi^{\ssup{\rm{DE}}}(\rho/\kappa_*)&\mbox{if }\gamma=1,\\
\chi^{\ssup{\rm{DB}}}_\gamma(\rho/\kappa_*)&\mbox{if }\gamma\ne1.
\end{cases}
\end{equation}
While the case $\gamma=1$ is qualitatively the same as the case of the double-exponential distribution analysed in \cite{GM98}, the case $\gamma\ne1$ shows a~new effect that was not present for constant diffusion speed $\kappa(t)\equiv1$. The diffusion is decelerated so strongly that the mass moves only by a~bounded amount.

\runinhead{Phase 3.} $\kh(t)/t \ll \kappa(t) \ll  t^{2/d}$.\\
The relation between ac-/deceleration and thickness of potential tails is so strong that the mass flows an unbounded amount of order $\alpha_t$ defined by~\eqref{alpha_def}. Since the acceleration is not too strong, the total mass comes from sites of extremely high potential values. Therefore, we get the continuous analogue to~\eqref{result}, but on scale $t\kappa(t)/\alpha_t^2$,
\begin{equation}
\label{resultcont}
\lim_{t\to\infty}\frac{\alpha_t^2}{t\kappa(t)}\log \erwxi{U(t)\exp(-\alpha_t^d H(t\alpha_t^{-d}))}
=-\begin{cases}
\chi^{\ssup{\rm{AB}}}(\rho)&\mbox{if }\gamma=1,\\
\chi^{\ssup{\rm{B}}}_\gamma(\rho)&\mbox{if }\gamma\ne1.
\end{cases}
\end{equation}
Hence, for $\gamma=1$ we are in the almost-bounded case \cite{HKM06} and for $\gamma<1$ in the bounded case \cite{BK01}. Note that we can have $\gamma\in[0,1+2/d)$ here, which has never been considered before in the parabolic Anderson model.

\runinhead{Phase 4.} $\kh(t)/t\ll \kappa(t) \asymp t^{2/d}$.\\
As in Phase 3, the mass flows an unbounded distance away from the origin. The acceleration reaches the critical level, such that this distance is of order~$t^{1/d}$, which is much larger than in Phase~3. Only so little mass reaches the sites in this large island that the potential is not extremely large here, but only by a bounded amount larger than the mean. Therefore, the characteristic variational formula does not only depend on the tails of the distribution, but on all values of the logarithmic moment generating function~$H$. This regime has strong connections to the large deviation result for a~random walk in random scenery model described in~\cite{GKS07}.

\runinhead{Phase 5.} $\kappa(t) \gg \kh(t)/t$ and $\kappa(t) \gg t^{2/d}$.\\
The speed is so high that, conjecturally, the values of the potential influence the expected total mass only via their mean, and the diffusion behaves like free Brownian motion with some diffusion constant that depends on the potential distribution. We will not present rigorous results for this phase in the present paper.

\smallskip

Note that, because of regular variation, $\kh(t)=t^{\gamma+o(1)}$. Hence, Phases~3 and~4 can only appear if we have $\gamma\le1+2/d$. The four universality classes for the constant-diffusion case $\kappa(t)\equiv 1$ are found in Phases~1--3 depending on whether $\gamma=1$ or~ $\gamma\ne 1$.

\subsection{Moment Asymptotics} \label{sec_asymptotik}

We now formulate our results. Recall the variational formulas defined in the Section~\ref{sec_vario} and set
\begin{equation*}
\chid_\gamma=\begin{cases}
        \chi^{\ssup{\rm{DE}}} & \textrm{if } \gamma=1, \\
        \chi^{\ssup{\rm{DB}}}_\gamma& \textrm{if } \gamma\ne 1,
      \end{cases}
\qquad\mbox{and}\qquad
\chic_\gamma=\begin{cases}
        \chi^{\ssup{\rm{AB}}} & \textrm{if } \gamma=1, \\
        \chi^{\ssup{\rm{B}}}_\gamma & \textrm{if } \gamma\ne 1.
      \end{cases}
\end{equation*}

Then we have the following result for the first three regimes of our model.
\begin{theorem}[Phase~1 -- Phase~3] \label{theo_phase1bis3}
Assume $\esssup \xi(0)\in\{0,\infty\}$.
\begin{enumerate}
 \item[\textup{(a)}]
If $\kappa(t)\ll \kh(t)/t$, then we have for $t\to\infty$
\begin{equation}
\label{asymp_phase1}
\erwxi{U(t)} = \exp\big(H(t)-2dt\kappa(t)(1+o(1))\big).
\end{equation}
 \item[\textup{(b)}]
If $\kappa(t)\asymp \kh(t)/t$, then
\begin{equation}
\label{asymp_phase2}
\erwxi{U(t)}=\exp\Big(H(t)-t\kappa(t)\chid_\gamma\Big(\frac{\rho}{\kappa_*}\Big)(1+o(1))\Big)
\end{equation}
with $\kappa_*=\lim_{t\to\infty}t\kappa(t)/\kh(t)\in(0,\infty)$.
\item[\textup{(c)}]
Let the assumption of Lemma~\textup{\ref{alpha_asymp}} hold, in particular $\kh(t)/t\ll \kappa(t)\ll t^{2/d}$. Furthermore suppose $\kh(t)\gg \log t$ and $\gamma<2$. Then
\begin{equation}
\label{asymp_phase3}
\erwxi{U(t)}=\exp\Big(\alpha_t^d H\Big(\frac{t}{\alpha_t^d}\Big)-\frac{t\kappa(t)}{\alpha_t^2}\chic_\gamma(\rho)(1+o(1))\Big).
\end{equation}
\end{enumerate}
\end{theorem}
Note that the assumption $\esssup \xi(0)\in\{0,\infty\}$ is not restrictive, since a shift of the potential would only lead to an additive constant in our results. The assumptions $\kh(t)\gg \log t$ and $\gamma<2$ in part~(c) of the theorem are purely technical, the first one only needed in the case~$\gamma=0$. Since~$\gamma<1+2/d$ in the respective phase (which follows from the assumption of Lemma~\textup{\ref{alpha_asymp}}), $\gamma<2$ is only a~restriction in dimension~$1$.

Now we come to Phase~4, where we will meet the variational formula $\chi^{\ssup{\rm{RWRS}}}_H(\theta)$ defined in~\eqref{chi_rwrs}. Since the result will no longer depend on the upper tails of the potential distribution, it will make sense to have an assumption for the expectation of~$\xi(0)$ instead of its essential supremum. Again, this is no loss of generality.
\begin{theorem}[Phase~4]
\label{theo_phase4}
Assume $\erwxi{\xi(0)}=0$ and $\kh(t)/t\ll\kappa(t)\asymp t^{2/d}$. Let $\gamma\in[0,1+2/d)$, $\gamma<2$. Then we have for $t\to\infty$
\begin{equation}
\label{asymp_phase4}
\erwxi{U(t)}=\exp\Big({-}t\kappa^*\chi^{\ssup{\rm{RWRS}}}_H\Big(\frac{1}{\kappa^*}\Big)(1+o(1))\Big)
\end{equation}
with $\kappa^*=\lim_{t\to\infty}\kappa(t)/t^{2/d}\in(0,\infty)$.
\end{theorem}

\subsection{Variational Convergence} \label{sec_gammakonv}

We now state a result which is both important in the proof of Theorem~\ref{theo_phase1bis3}(c) and of independent interest as a~connection between the discrete variational formula~$\chid_\gamma(\rho)$ and its continuous analogue~$\chic_\gamma(\rho)$. In the case~$\gamma=1$, this fact is stated in \cite{HKM06} and is derived without difficulties from an explicit representation of $\chi^{\ssup{\rm{AB}}}(\rho)$. The proof for the case~$\gamma\ne 1$ is much more involved and uses techniques from the theory of finite elements.
\begin{proposition}
\label{gammakonvergenz}
Let~$\rho>0$. As~$\kappa\to\infty$, we have
\begin{equation}
\label{gammakonvergenz_de}
\kappa\chi^{\ssup{\rm{DE}}}\Big(\frac{\rho}{\kappa}\Big)=\chi^{\ssup{\rm{AB}}}(\rho)+\rho\frac{d}{2}\log\kappa + o(1)
\end{equation}
and for $\gamma\in [0,1+2/d)\setminus\{1\}$
\begin{equation}
\label{gammakonvergenz_db}
\kappa^{1-d\nu}\chi^{\ssup{\rm{DB}}}_\gamma\Big(\frac{\rho}{\kappa}\Big)
=\chi^{\ssup{\rm B}}_\gamma(\rho)+\rho\frac{1-\kappa^{-d\nu}}{1-\gamma}+o(1)
\end{equation}
with $\nu=\frac{1-\gamma}{2+d(1-\gamma)}$.
\end{proposition}

Note that \eqref{gammakonvergenz_de} and \eqref{gammakonvergenz_db} are consistent, as \eqref{gammakonvergenz_de} is a continuous continuation of \eqref{gammakonvergenz_db} to $\gamma=1$. Proposition~\ref{gammakonvergenz} shows that Phases 2 and 3 can be continuously transformed into each other, i.e., the transition between them is actually no phase transition in the sense of statistical mechanics.

\section{Proof of Variational Convergence (Proposition~\textup{\ref{gammakonvergenz}})}\label{sec_proofs}

The asymptotics \eqref{gammakonvergenz_de} follows from the arguments in \cite[p.~313]{HKM06}. To show~\eqref{gammakonvergenz_db}, we remark first that the summand $\frac{\rho}{1-\gamma}$ drops out in both \eqref{chi_b} and~\eqref{chi_db}. Therefore \eqref{gammakonvergenz_db} is equivalent to
\begin{equation}
\label{gammakonvergenz_db_2}
\lim_{\kappa\to\infty}\kappa^{1-d\nu}\inf_{p\in\masz(\zd)}
        \Big\{{-}\scal{\Delta\sqrt{p}}{\sqrt{p}} + \frac{\rho}{\kappa(1-\gamma)}\sum_{z\in\zd}p(z)^{\gamma}\Big\}
=\hat{\chi}_\gamma(\rho),
\end{equation}
where
$$
\hat{\chi}_\gamma(\rho)=\inf_{\substack{g\in \mathrm{H}^1(\rd)\\ \norm{g}_2=1}}
        \Big\{\int_{\rd} |\nabla g|^2 + \frac{\rho}{1-\gamma}\int_{\rd} g^{2\gamma} \Big\}.
$$
The proof of the upper bound of \eqref{gammakonvergenz_db_2} is standard and we will here only give the idea. To an approximate minimiser~$g$ for the infimum in~$\hat{\chi}_\gamma(\rho)$ and for small $\varepsilon>0$, we define a probability measure~$p_\varepsilon$ by
$$
p_\varepsilon(z)=\int_{\varepsilon z+[0,\varepsilon)^d} g(x)^2 \,\D x,\quad z\in\zd.
$$
Assuming that~$g$ is smooth and compactly supported, we can make use of Taylor expansions to see that, as $\varepsilon\downarrow 0$,
$$
-\varepsilon^{-2}\scal{\Delta\sqrt{p_\varepsilon}}{\sqrt{p_\varepsilon}} \to \int_{\rd}|\nabla g|^2
\quad\mbox{and}\quad \varepsilon^{d(1-\gamma)}\sum_{z\in \zd}p_\varepsilon(z)^\gamma\to \int_{\rd}g^{2\gamma}.
$$
Recall $\gamma<1+2/d$. Putting $\varepsilon=\kappa^{-(1-d\nu)/2}=\kappa^{-1/(2+d(1-\gamma))}\downarrow 0$ as $\kappa\to\infty$, this shows the upper bound.

Let us now turn to the lower bound. This proof is pretty involved and comes in several steps. The principal idea and main arguments are taken from \cite[Proof of~(5.3)]{HKM06}. However, we could not find an argument for the $\mathrm{L}^2$-normalisation of the limit function in their approximation approach, since this involves interchanging integral and limit, which seems to be hard to justify. Hence, we use a different construction. Furthermore, our consideration of $\gamma>1$ causes some additional difficulties.

We will only treat the case $\gamma>1$. The structure for $\gamma<1$ is similar, for details we refer to the proofs of \cite[Prop.~3.4.7 and Prop.~5.2.1]{S10}. We denote $S(p)=-\scal{\Delta\sqrt{p}}{\sqrt{p}}$.

\runinhead{Step 1.}  We choose minimising sequences $\kappa_n\to\infty$ and $(p_n)_n$ from $\masz(\zd)$ for the left hand side of~\eqref{gammakonvergenz_db_2}. Put $a_n=\kappa_n^{(1-d\nu)/2}$. We now argue that we can assume, without loss of generality, that
\begin{equation}
\label{finite}
\sup_{n\in\nn}a_n^2S(p_n)<\infty.
\end{equation}
For this, we need the following discrete Sobolev inequality:
\begin{lemma}
\label{sobolev_discrete}
Let $\gamma>1$ with $\gamma(d-2)<d$. There exists a constant $c=c_{d,\gamma}$ such that for all $p\in\masz(\zd)$
$$
\sum_{z\in\zd}p(z)^\gamma\le c S(p)^{d(\gamma-1)/2}.
$$
\end{lemma}
\begin{proof}
See \cite[Lemma~3.2.10]{S10}.
\end{proof}
Now suppose that~\eqref{finite} does not hold. Then, by Lemma~\ref{sobolev_discrete} and because of $d(\gamma-1)/2<1$,
\begin{multline*}
\lim_{n\to\infty}a_n^2\bigg\{S(p_n)+\frac{\rho}{a_n^{2+d(1-\gamma)}(1-\gamma)}\sum_{z\in\zd}p_n(z)^\gamma\bigg\}\\
\ge\limsup_{n\to\infty}\Big\{a_n^2S(p_n)-\frac{c\rho}{\gamma-1}\big(a_n^2S(p_n)\big)^{d(\gamma-1)/2}\Big\}
=\infty.
\end{multline*}
Since $(p_n)_n$ is a minimising sequence, the lower bound would now be trivially satisfied. Hence, we can assume~\eqref{finite}.

\runinhead{Step 2.} We compactify on a box $B_{Ra_n}=[-Ra_n,Ra_n]^d\cap\zd$ for $R>0$. Consider the periodised probability measures
$$
p_n^R(z)=\sum_{k\in(2Ra_n+1)\zd}p_n(z+k),\quad z\in B_{Ra_n}.
$$
In \cite[Lemma~1.10]{GM98}, it was shown that $S^{\pi,R}(p_n^R)\le S(p_n)$ in the one-dimensional case, where $S^{\pi,R}$ is the Dirichlet form with periodic boundary condition. This holds as well in higher dimensions, besides we have $\frac{1}{1-\gamma}\sum_{z\in B_{Ra_n}}p_n^R(z)^\gamma\le\frac{1}{1-\gamma}\sum_{z\in\zd}p_n(z)^\gamma$ by subadditivity. Therefore it will be sufficient to prove that
\begin{equation}
\label{gammakonvergenz_db_3}
\liminf_{R\to\infty}\liminf_{n\to\infty}
a_n^2\bigg\{S^{\pi,Ra_n}\big(p_n^{R}\big)+\frac{\rho}{1-\gamma}a_n^{-2-d(1-\gamma)}\sum_{z\in B_{Ra_n}}\big(p_n^{R}(z)\big)^\gamma\bigg\}
    \ge \hat{\chi}_\gamma(\rho).
\end{equation}
Since $S^{\pi,R}(p_n^R)\le S(p_n)$, \eqref{finite} implies
\begin{equation}
\label{finite_per}
\sup_{n\in\nn}a_n^2S^{\pi,R}(p_n^R)<\infty.
\end{equation}

\runinhead{Step 3.} Our goal is to construct potential minimisers for $\hat{\chi}_\gamma(\rho)$ that interpolate the values of the rescaled step functions $h_n(x)=\sqrt{a_n^dp_n^R(\floor{a_nx})}$ on the lattice $\{x=z/a_n:\, z\in B_{Ra_n}\}$. In the present step, we define piecewise linear interpolations~$g_n\in\mathrm{H}^1(Q_R^{\ssup{n}})$ with $Q_R^{\ssup{n}}=[-R,R+a_n^{-1})^d$, which we will slightly modify in Step~4 in order to obtain normalised $\mathrm{H}^1(\rd)$-functions.

We borrow a technique from finite elements theory, see e.g.\ \cite{B07}. Consider the triangulation
$$
Q_R^{\ssup{n}}=\bigcup_{z\in B_{Ra_n}}\bigcup_{\sigma\in\mathfrak{S}_d}T_\sigma(z),
$$
where $\mathfrak{S}_d$ is the set of permutations of $1,\ldots,d$ and $T_\sigma(z)$ is the $d$-dimensional tetrahedron defined as the convex hull of the points $z,z+e_{\sigma(1)},\ldots,z+e_{\sigma(1)}+\cdots+e_{\sigma(d)}$, where $e_i$ is the $i$-th unit vector in~$\rd$. Note that the tetrahedra are disjoint up to the boundary. On each tetrahedron~$T_\sigma(z)$, we define a function
$$
g_{n,z,\sigma}(x)=b_{n,z,\sigma}^{\ssup{0}}+\sum_{k=1}^d b_{n,z,\sigma}^{\ssup{k}}(a_n x_{\sigma(k)}-z_{\sigma(k)}), \quad
x=(x_1,\ldots,x_d)\in T_\sigma(z),
$$
where the coefficients are given by
\begin{align*}
b_{n,z,\sigma}^{\ssup{0}}&=\sqrt{a_n^dp_n^R(z)}=h_n\Big(\frac{z}{a_n}\Big),\\
b_{n,z,\sigma}^{\ssup{k}}&=\sqrt{a_n^dp_n^R(z+e_{\sigma(1)}+\cdots+e_{\sigma(k)})}-\sqrt{a_n^dp_n^R(z+e_{\sigma(1)}+\cdots+e_{\sigma(k-1)})}
\end{align*}
for $k=1,\ldots,d$, where $p_n^R$ is continued periodically outside $B_{Ra_n}$. Then $g_{n,z,\sigma}$ satisfies
\begin{equation*}
g_{n,z,\sigma}(\tilde{z}/a_n)=h_n(\tilde{z}/a_n) \quad \text{for all }\tilde{z}\in T_\sigma(z)\cap\zd.
\end{equation*}
The values of all functions $g_{n,z,\sigma}$ on the common borders of their respective tetrahedra coincide; see \cite[Proof of Lemma 2.1]{BK10} for a detailed argument. Hence, the function $g_n\colon Q_R^{\ssup{n}}\to\rr$ given by
$$
g_n(x)=g_{n,z,\sigma}(x) \quad \text{if } x\in T_\sigma(z)
$$
is well-defined and continuous, and $g_n\in\mathrm{H}^1(Q_R^{\ssup{n}})$.

A direct calculation for the gradient gives $\partial_{x_{\sigma(k)}} g_n(x)=a_n b_{n,z,\sigma}^{\ssup{k}}$ and thus
\begin{equation}
\label{gradient}
\int_{Q_R^{\ssup{n}}}|\nabla g_n|^2 = a_n^2S^{\pi,R}(p_n^R).
\end{equation}
Note that by~\eqref{finite_per} this is bounded in $n$. Now consider the $\mathrm{L}^2$-norm of~$g$. Because of $|a_nx_{\sigma(k)}-z_{\sigma(k)}|\le 1$ and $b_{n,z,\sigma}^{\ssup{k}}=a_n^{-1}\partial x_{\sigma(k)} g_n(x)$ we obtain
$$
\norm{(g_n-h_n)\1_{Q_R^{\ssup{n}}}}_2^2
    \le a_n^{-2} \int_{Q_R^{\ssup{n}}}
            \Big(\sum_{i=1}^d\frac{\partial}{\partial x_i} g_n(x)\Big)^2 \,\D x.
$$
By Jensen's inequality, $(\sum_{i=1}^d c_i)^2\le d \sum_{i=1}^d c_i^2$. Since $\norm{h_n\1_{Q_R^{\ssup{n}}}}_2=1$, the triangle inequality gives
\begin{equation}
\label{l2norm}
|\norm{g_n\1_{Q_R^{\ssup{n}}}}_2-1|^2\le d a_n^{-2}\int_{Q_R^{\ssup{n}}}|\nabla g_n|^2,
\end{equation}
which tends to zero as $n\to\infty$ by~\eqref{gradient} and~\eqref{finite_per}.

A similar calculation for the $\mathrm{L}^{2\gamma}$-norm results in
$$
\norm{(g_n-h_n)\1_{Q_R^{\ssup{n}}}}_{2\gamma}^{2\gamma}
    \le d^\gamma a_n^{-2\gamma}\int_{Q_R^{\ssup{n}}}|\nabla g_n|^{2\gamma}.
$$
Because of $p_n^R(z)\in[0,1]$, we have $|b_{n,z,\sigma}^{\ssup{k}}|\le a_n^{d/2}$ and therefore $|\nabla g_n|^2\le da_n^{d+2}$. For $\gamma>1$, this yields
$$
|\nabla g_{n}(x)|^{2\gamma}=d^\gamma a_n^{(d+2)\gamma}\bigg(\frac{|\nabla g_{n}(x)|^2}{da_n^{d+2}}\bigg)^\gamma
        \le d^{\gamma-1} a_n^{2\gamma} a_n^{-2-d(1-\gamma)}|\nabla g_{n}(x)|^2.
$$
Now use triangle inequality to get
\begin{equation}
\label{gammanorm}
a_n^{d(\gamma-1)}\sum_{z\in B_{Ra_n}}\big(p_n^R(z)\big)^\gamma=\norm{h_n\1_{Q_R^{\ssup{n}}}}_{2\gamma}^{2\gamma}
\le \Big(\norm{g_n\1_{Q_R^{\ssup{n}}}}_{2\gamma}+c_n a_n^{\frac{-2-d(1-\gamma)}{2\gamma}}\Big)^{2\gamma},
\end{equation}
where $c_n=(d^{2\gamma-1}\int_{Q_R^{\ssup{n}}}|\nabla g_n|^{2})^{1/(2\gamma)}$ is bounded in $n$.

\runinhead{Step 4.} In order to adapt our function~$g_n$ to zero boundary conditions, we introduce a cut off function $\Psi_R(x)=\prod_{i=1}^d\psi_R(x_i)$, ${x=(x_1,\ldots,x_d)\in\rd}$, where $\psi_R=1$ on $[-R+\sqrt{R},R-\sqrt{R}]$, $\psi_R=0$ on $\rr\setminus[-R,R]$ and it interpolates linearly in-between. Then $0\le \psi_R\le 1$ and $|\psi_R'|\le 1/\sqrt{R}$. Let us estimate the relevant terms for the $\mathrm{H}^1(\rd)$-function $g_n\Psi_R$ (which is zero outside~$Q_R=[-R,R]^d$). As for the gradient,
\begin{align*}
\int_{\rd}\Big(\frac{\partial}{\partial x_i}(g_n\Psi_R)(x)\Big)^2\,\D x
&\le \int_{Q_R}\Big(\frac{\partial}{\partial x_i}g_n(x)\Big)^2\,\D x+\frac{1}{R}\int_{Q_R}g_n(x)^2\,\D x \\
&\quad +\frac{2}{\sqrt{R}}\sqrt{\int_{Q_R}\Big(\frac{\partial}{\partial x_i}g_n(x)\Big)^2\,\D x}
                       \sqrt{\int_{Q_R}g_n(x)^2\,\D x},
\end{align*}
where we used the properties of~$\psi_R$ and the Cauchy--Schwarz-inequality. Since all integrals are bounded (recall \eqref{gradient}, \eqref{l2norm} and \eqref{finite_per}), we find a~constant $c>0$ such that for all~$n$ and all $R$
\begin{equation}
\label{gradient_nrb}
\int_{\rd}|\nabla(g_n\Psi_R)(x)|^2\,\D x \le \int_{Q_R}|\nabla g_n(x)|^2\,\D x + \frac{c}{\sqrt{R}}.
\end{equation}

Our basic tool for estimating the $\mathrm{L}^2$- and $\mathrm{L}^{2\gamma}$-norm of $g_n\Psi_R$ is a variation of the shift lemma \cite[Lemma~3.4]{DV75}. Indeed, using the shift-invariance of the variational
problem because of periodic boundary conditions, the mass of a~nonnegative function on the boundary $Q_R\setminus Q_{R-\sqrt{R}}$ can, after suitable shifting, be estimated by its total mass on $Q_R$ times the quotient of the volumes. Applying this to $g_n^2+g_n^{2\gamma}$, we may assume that
$$
\int_{Q_R\setminus Q_{R-\sqrt{R}}} \big(g_n^2+g_n^{2\gamma}\big)
     \le \frac{d}{\sqrt{R}}\int_{Q_R} \big(g_n^2+g_n^{2\gamma}\big).
$$
Skipping the details, this leads to
\begin{equation}
\label{gammanorm_nrb}
\norm{g_n\Psi_R}_{2\gamma}^{2\gamma}\ge \Big(1-\frac{d}{\sqrt{R}}\Big)\norm{g_n\1_{Q_R^{\ssup{n}}}}_{2\gamma}^{2\gamma}-\frac{c}{\sqrt{R}}
\end{equation}
and, with use of \eqref{l2norm},
\begin{equation}
\label{l2norm_nrb}
|\norm{g_n\Psi_R}_{2}^{2}-1|\le \frac{c}{\sqrt{R}}
\end{equation}
for a suitable constant, not depending on $n$ or $R$, which we also denote $c>0$.

\runinhead{Step 5.} Now we put everything together to show~\eqref{gammakonvergenz_db_3}.
We use \eqref{gradient} and \eqref{gammanorm} and note that $1<\gamma<1+2/d$ to get
\begin{align*}
& \liminf_{n\to\infty}a_n^2\bigg\{S^{\pi,Ra_n}\big(p_n^{R}\big)
                        +\frac{\rho}{1-\gamma}a_n^{-2-d(1-\gamma)}\sum_{z\in B_{Ra_n}}\big(p_n^{R}(z)\big)^\gamma\bigg\}\\
& \ge \limsup_{n\to\infty}\Big(\int_{Q_R^{\ssup{n}}}|\nabla g_n|^2
    -\frac{\rho}{\gamma-1}\norm{g_n\1_{Q_R^{\ssup{n}}}}_{2\gamma}^{2\gamma}\Big).
\end{align*}
Next, we plug in \eqref{gradient_nrb} and \eqref{gammanorm_nrb} obtaining
\begin{align*}
& \liminf_{R\to\infty}\liminf_{n\to\infty}a_n^2\bigg\{S^{\pi,Ra_n}\big(p_n^{R}\big)
                        +\frac{\rho}{1-\gamma}a_n^{-2-d(1-\gamma)}\sum_{z\in B_{Ra_n}}\big(p_n^{R}(z)\big)^\gamma\bigg\}\\
& \ge \limsup_{R\to\infty}\limsup_{n\to\infty}\Big(\int_{\rd}|\nabla(g_n\Psi_R)|^2
    -\frac{\rho}{\gamma-1}\norm{g_n\Psi_R}_{2\gamma}^{2\gamma}\Big).
\end{align*}
With the help of \eqref{l2norm_nrb}, we can replace $g_n\Psi_R$ by its normalised version $g_n\Psi_R/\norm{g_n\Psi_R}_2$, which is a~candidate for the infimum in~$\hat{\chi}_\gamma(\rho)$. This yields the assertion.

\section{Proof for Phases 1--3 (Theorem~\textup{\ref{theo_phase1bis3}})}\label{sec-Proof1-3}

The proof of~(a) and~(b) is analogous to the proof of \cite[Theorem~1.2]{GM98} (see \cite{S10} for details), therefore we only sketch the idea here and omit all details, like compactification, cutting, or error terms.

Denote by $\ell_t(z)=\int_0^t\1_{\{X_s=z\}}\,\D s$ the local time of the random walk path~$(X_s)_{s\in[0,t]}$ with generator $2d\kappa(t)\Delta$ in the point $z\in\zd$. Starting from the Feynman--Kac formula~\eqref{FKform}, we apply the asymptotics~\eqref{konvergenz_h} to the normalised local times $\ell_t/t$. Heuristically, this gives
$$
\erwxi{U(t)}\E^{-H(t)} \approx \ee^{\ssup{t}}_0\Big[{\exp\Big(\kh(t)\sum_{z\in \zd}
                \rho\hat{H}\Big(\frac{\ell_t(z)}{t}\Big) (1+o(1)) \Big)}\Big], \quad t\to\infty.
$$
Denote by $\pp^{\ssup{t}}_{0}$ the probability measure related to~$\ee^{\ssup{t}}_{0}$. Under~$\pp^{\ssup{t}}_{0}$, the process $(\ell_t/t)_t$ satisfies a~large deviation principle on scale~$t\kappa(t)$ with rate function $p\mapsto -\scal{\Delta\sqrt{p}}{\sqrt{p}}$. In part~(a), the scale~$\kh(t)$ is asymptotically smaller than~$t\kappa(t)$, therefore the main contribution comes from the event that the process~$(X_s)_{s\in[0,t]}$ stays in the origin, which leads to formula~\eqref{asymp_phase1}. In part~(b), because of $\kh(t)\asymp t\kappa(t)$, an application of Varadhan's lemma gives~\eqref{asymp_phase2}.

The proof of~(c) follows mainly the arguments of \cite{HKM06} (who consider only~$\gamma=1$), adapting them to the new scale~$t\kappa(t)/\alpha_t^2$. The case $\gamma<1$ was treated in a~similar way in \cite{BK01}, whereas the case $\gamma>1$ did not appear originally in Phase~3. For convenience, we give a~universal derivation for all values~$\gamma\in[0,1+2/d)$.

 By an adaption of \cite[Prop.~3.4]{HKM06}, the rescaled and normalised local times
\begin{equation}
\label{localtime}
L_t(y)=\frac{\alpha_t^d}{t}\ell_t(\floor{\alpha_t y}), \quad y\in\rd,
\end{equation}
with $\alpha_t$ defined by~\eqref{alpha_def}, satisfy under $\pp^{\ssup{t}}_0(\,\cdot\,\1_{\{\supp L_t\subset Q_R\}})$ a~large deviation principle in the weak topology induced by test integrals against continuous functions, where we recall that $Q_R=[-R,R]^d$. The scale of the principle is $t\kappa(t)/\alpha_t^2$ and the rate function is $g^2\mapsto \int_{\rd}|\nabla g|^2$ for $g\in\mathrm{H}^1(\rd)$ with $\supp g\subset Q_R$ and $\norm{g}_2=1$.

For a~lower bound, we start again with~\eqref{FKform} and insert the indicator on the event $\{\supp L_t\subset Q_R\}$, using the notation $\ee_{0,R}^{\ssup{t}}[\,\cdot\,]$. After transforming
\begin{align}
\erwxi{U(t)}
    &\ge\ee_{0,R}^{\ssup{t}} \Big[\exp\Big(\sum_{z\in \zd}H(\ell_t(z))\Big)\Big]\nonumber\\
    &= \ee_{0,R}^{\ssup{t}} \Big[\exp\Big(\alpha_t^d\int_{Q_R}H\Big(\frac{t}{\alpha_t^{d}}L_t(y)\Big)\,\D y\Big)\Big]\nonumber\\
    &=\E^{\alpha_t^dH(\frac{t}{\alpha_t^{d}})}
        \ee_{0,R}^{\ssup{t}} \Big[\exp\Big(\frac{t\kappa(t)}{\alpha_t^2}
                    \int_{Q_R}\frac{H\big(\frac{t}{\alpha_t^{d}}L_t(y)\big)-L_t(y)H\big(\frac{t}{\alpha_t^{d}}\big)}
                    {\kh\big(\frac{t}{\alpha_t^{d}}\big)}\,\D y\Big)\Big],
\label{einzwischenschritt}
\end{align}
we restrict the integral to the part where $L_t(y)\le M$ for some $M>1$, noting that the integrand on the set $\{L_t(y)>M\}$ is nonnegative because of the convexity of~$H$. Then we apply the locally uniform asymptotics~\eqref{konvergenz_h}. Next, to get rid of the indicator on $\{L_t(y)\le M\}$, we introduce a~H\"older parameter $\eta\in(0,1)$ to separate the expectations over the whole integral and over the difference set $\{L_t(y)>M\}$. The expectation over the rest term can be shown to be negligible on the exponential scale $t\kappa(t)/\alpha_t^2$ (see \cite[pp.~86f]{S10}; here we use Lemma~\ref{regvar_h}(a) and the assumption that~$\gamma<2$). Finally, we apply the large deviation principle for $L_t$ and Varadhan's lemma; the lower semi-continuity of $g^2\mapsto\int_{Q_R}\hat{H}\circ g^2$ was proved in \cite[Lemma~3.5]{HKM06} for~$\gamma=1$ and can be shown similarly for all positive~$\gamma$. Summarizing, we obtain for $\gamma>0$\allowdisplaybreaks
\begin{align*}
&\liminf_{t\to\infty}\frac{\alpha_t^2}{t\kappa(t)}\log\big(\erwxi{U(t)}\E^{-\alpha_t^dH(t\alpha_t^{-d})}\big) \\
    &\ge \liminf_{M\to\infty}\liminf_{t\to\infty}\frac{\alpha_t^2}{t\kappa(t)}\log \ee_{0,R}^{\ssup{t}}\Big[%
            \exp\Big(\frac{t\kappa(t)}{\alpha_t^2}
             \int_{Q_R}\rho\hat{H}(L_t(y))\1_{\{L_t(y)\le M\}}\,\D y\,  \Big)\Big] \\
&\ge\liminf_{t\to\infty}\frac{\alpha_t^2}{t\kappa(t)}\log \ee_{0,R}^{\ssup{t}}\Big[%
     \exp\Big((1-\eta)\frac{t\kappa(t)}{\alpha_t^2}
             \int_{Q_R}\rho\hat{H}(L_t(y))\,\D y\,  \Big)\Big] \\
&\ge -\inf_{\substack{ g\in \mathrm{H}^1(\rd)\\ \supp g\subset Q_R\\ \norm{g}_2=1}}
        \Big\{\int_{Q_R}|\nabla g|^2- \rho(1-\eta)\int_{Q_R} \hat{H}\circ g^2  \Big\}.
\end{align*}
A standard argument shows that the compactified variational formula converges to~$\chic_\gamma(\rho)$ as $R\to\infty$ and $\eta\downarrow 0$. For the case~$\gamma=0$, we refer to \cite[pp.~85f]{S10}.

Now we prove the upper bound of~\eqref{asymp_phase3}. For technical reasons, we will not work with the large deviation principle, but use a~method derived in \cite{BHK07}. First, we compactify with the help of an eigenvalue expansion described in \cite{BK01} and applied in \cite{HKM06}. Replacing carefully $t$ by~$t\kappa(t)$ in their proofs, we find for $R>0$
\begin{equation}
\label{compact_upperbound}
\frac{\alpha_t^2}{t\kappa(t)}\log\erwxi{U(t)}\le \frac{C}{R^2}+\frac{\alpha_t^2}{t\kappa(t)}\log\erwxi{U_{{4R\alpha_t}}(t)}+o(1),
    \quad t\to\infty,
\end{equation}
with some constant $C>0$, where $U_{{R\alpha_t}}(t)=\ee_{0,R}^{\ssup{t}}[\E^{\int_0^t \xi(X_s)\,\D s}]$. Similarly to~\eqref{einzwischenschritt}, we can write
$$
\erwxi{U_{{R\alpha_t}}(t)}=\E^{\alpha_t^dH(\frac{t}{\alpha_t^{d}})}
\ee_{0,R}^{\ssup{t}} \Big[\exp\Big(\frac{t\kappa(t)}{\alpha_t^2}
\!\sum_{z\in B_{R\alpha_t}}\!\!\frac{H(\ell_t(z))-\frac{t}{\alpha_t^{d}}\ell_t(z)H\big(\frac{t}{\alpha_t^{d}}\big)}
                    {\kh\big(\frac{t}{\alpha_t^{d}}\big)}\Big)\Big],
$$
where we recall that $B_R=[-R,R]\cap\zd$.
We split the sum into the part where $\ell_t(z)\le Mt\alpha_t^{-d}$ and the rest where $\ell_t(z)> Mt\alpha_t^{-d}$ for some $M>1$, separating the respective expectations with H\"older's inequality. The rest term can again be neglected on the exponential scale~$t\kappa(t)/\alpha_t^2$, while an application of~\eqref{konvergenz_h} in the main term leads to
\begin{align}
&\limsup_{t\to\infty}\frac{\alpha_t^2}{t\kappa(t)}\log \Big( \erwxi{U_{{R\alpha_t}}(t)}\,\E^{-\alpha_t^dH(\frac{t}{\alpha_t^{d}})} \Big) \nonumber\\
&\le \limsup_{t\to\infty}\frac{\alpha_t^2}{t\kappa(t)}\log \ee_{0,R}^{\ssup{t}} \Big[\exp\Big(\tilde{\rho}\,\frac{t\kappa(t)}{\alpha_t^{d+2}}
                    \sum_{z\in B_{{R\alpha_t}}}\hat{H}\Big(\frac{\alpha_t^d}{t}\ell_t(z)\Big)
                                \1_{\{\ell_t(z)\le M\frac{t}{\alpha_t^d}\}}\Big)\Big],
\label{nocheinzwischenschritt}
\end{align}
where $\tilde{\rho}=\rho(1+\eta)$ with the H\"older parameter~$\eta\in(0,1)$. Next, we can omit the indicator on the event $\{\ell_t(z)\le M t \alpha_t^{-d}\}$ noting that the function~$\hat{H}$ is nonnegative on $[1,\infty)$.

We now need the mentioned tool from \cite{BHK07}, namely an explicit description of the local times density, which provides an upper bound on exponential functionals like in~\eqref{nocheinzwischenschritt} in the form of a~variational formula: Define
$$
G_t(p)=\alpha_t^{-(d+2)}\sum_{z\in \zd}\hat{H}(\alpha_t^d p(z))
$$
for $p\in\masz(\zd)$. Then, noting that our local times are related to a~random walk with generator~$2d\kappa(t)\Delta$, a~respective adaption in the formulation of \cite[Prop.~3.3]{HKM06} gives
\begin{align*}
&\ee_{0,R}^{\ssup{t}}\Big[\exp\Big(t\kappa(t)\tilde{\rho}\,G_t\big(\frac{\ell_t}{t}\big)\Big)\Big]\\
&\le \exp\Big(t\kappa(t)\sup_{\substack{p\in\masz(\zd)\\ \supp p\subset B_{{R\alpha_t}}}}                                        \big\{\tilde{\rho}\,G_t(p)+\scal{\Delta\sqrt{p}}{\sqrt{p}}\big\}\Big)(2dt\kappa(t))^{|B_{{R\alpha_t}}|}|B_{{R\alpha_t}}|\\
&\le \exp\Big({-}\frac{t\kappa(t)}{\alpha_t^2} \chi_t(\tilde{\rho})\Big) \,\E^{o(t\kappa(t)/\alpha_t^2)},
\end{align*}
where we put
$$
\chi_{t}(\tilde{\rho})=-\alpha_t^2\sup_{p\in\masz(\zd)}\big\{\tilde{\rho}\,G_t(p)+\scal{\Delta\sqrt{p}}{\sqrt{p}}\big\}.
$$
In the last step, we also used the properties of the scale function~$\alpha_t$ mentioned in Lemma~\ref{alpha_asymp} and the assumption $\kh(t)\gg\log t$.

Now a direct calculation shows that
$$
\chi_{t}(\tilde{\rho})=\begin{cases}
\alpha_t^2\chi^{\ssup{\rm{DE}}}\Big(\dfrac{\tilde{\rho}}{\alpha_t^2}\Big)+\tilde{\rho}\dfrac{d}{2}\log\alpha_t^2
        & \text{for } \gamma=1,\\[2mm]
\alpha_t^2\chi^{\ssup{\rm{DB}}}_\gamma\Big(\dfrac{\tilde{\rho}}{\alpha_t^{2+d(1-\gamma)}}\Big)+\tilde{\rho}\dfrac{1-\alpha_t^{-d(1-\gamma)}}{1-\gamma}
        & \text{for } \gamma\ne 1.
\end{cases}
$$
In both cases, we can apply Prop.~\ref{gammakonvergenz} with $\kappa=\alpha_t^{2+d(1-\gamma)}\to\infty$ for $t\to\infty$, since $\gamma<1+2/d$. Hence, $\chi_{t}(\tilde{\rho})$~converges to $\chi^{\ssup{\rm AB}}(\tilde{\rho})$ in the case~$\gamma=1$ and to $\chi^{\ssup{\rm B}}_\gamma(\tilde{\rho})$ in the case~$\gamma\ne 1$, i.e.~to~$\chic_\gamma(\tilde{\rho})$ in both cases. In summary, \eqref{nocheinzwischenschritt} becomes
\begin{align*}
&\limsup_{t\to\infty}\frac{\alpha_t^2}{t\kappa(t)}\log \Big(\erwxi{U_{{R\alpha_t}}(t)}\,\E^{-\alpha_t^dH(\frac{t}{\alpha_t^{d}})}\Big)
\le \limsup_{t\to\infty}(-\chi_{t}(\tilde{\rho}))
\le -\chic_\gamma(\tilde{\rho}).
\end{align*}
By a scaling argument, one can see that $\chic_\gamma(\tilde{\rho})=\chic_\gamma(\rho(1+\eta))$ converges to $\chic_\gamma(\rho)$ for $\eta\downarrow 0$.
Together with~\eqref{compact_upperbound}, the assertion~\eqref{asymp_phase3} is thus shown, which finishes the proof of Theorem~\ref{theo_phase1bis3}.

\section{Proof for Phase 4 (Theorem~\textup{\ref{theo_phase4}})}\label{sec-Proof4}

Phase~4 is characterised by the fact that the space--time scale ratio is constant: $\alpha_t=t^{1/d}$, i.e.~$t/\alpha_t^d=1$. We rescale both local times and potential,
$$
L_t(y)=\ell_t(\floor{\alpha_t y}) \quad \text{and} \quad \bar{\xi}_t(y)=\xi(\floor{\alpha_t y}), \quad y\in\rd.
$$
Note that because of $\kappa(t)=\kappa^*t^{2/d}(1+o(1))$, the definition of the rescaled (and normalised) local times is asymptotically equivalent to~\eqref{localtime}, hence we have again an LDP under $\pp^{\ssup{t}}_0(\,\cdot\,\1_{\{\supp L_t\subset Q_R\}})$ on scale $t\kappa(t)/\alpha_t^2\asymp t$ with rate function $g^2\mapsto \int_{\rd}|\nabla g|^2$ for $g\in\mathrm{H}^1(\rd)$ satisfying $\supp g\subset Q_R$ and $\norm{g}_2=1$.

We will frequently make use of arguments from \cite{GKS07}, in particular their main result on large deviations for the scalar product $\scal{L_t}{\bar{\xi}_t}$. The time parameter~$t$ in \cite{GKS07} is replaced by~$t\kappa(t)$ and our scale function~$\alpha_t=t^{1/d}$ corresponds to the \cite{GKS07}-scale at time~$t\kappa(t)$, multiplied by $(\kappa^*)^{-1/(d+2)}$. Thus, \cite[Thm.~1.3]{GKS07} reads
\begin{equation}
\label{thm_gks07}
\lim_{t\to\infty}\frac{1}{t}\log \pp_0^{\ssup{t}}\times\mathrm{Prob}\big(\scal{L_t}{\bar{\xi}_t}>u\big)=-(\kappa^*)^{d/(d+2)}\chi^{[\mathrm{GKS07}]}_H(u)
\end{equation}
for $u>0$ such that $u\in (\supp \xi(0))^\circ$, where $\mathrm{Prob}$ is the probability with respect to the potential and
$$
\chi^{[\mathrm{GKS07}]}_H(u)=\inf_{\substack{g\in \mathrm{H}^1(\rd)\\ \norm{g}_2=1}}\Big\{%
    \int_{\rd}|\nabla g(y)|^2\,\D y + \sup_{\beta>0}\big[\beta u - \int_{\rd}H(\beta g^2(y))\,\D y\big] \Big\}.
$$
By rescaling and duality, it turns out that the variational problem~$\chi^{\mathrm{(RWRS)}}_H$ that we wish to find in this proof is essentially the negative Legendre transform of $\chi^{[\mathrm{GKS07}]}_H$:
\begin{equation}
\label{gks07_rwrs}
\sup_{u>0}\big\{\beta u-\chi^{[\mathrm{GKS07}]}_H(u)\big\}=-\beta^{-2/d}\chi^{\mathrm{(RWRS)}}_H\big(\beta^{1+2/d}\big),\qquad \beta>0.
\end{equation}

Let us come to the lower bound of~\eqref{asymp_phase4}. A transformation of the Feyn\-man--Kac formula~\eqref{FKform} gives
$$
\erwxi{U(t)}=\erwxi{\ee_0^{\ssup{t}}\big[{\exp\big(t\scal{L_t}{\bar{\xi}_t}\big)}\big]}
            =\int_{\rr} t\E^{ut}\,\pp_0^{\ssup{t}}\times\mathrm{Prob}\big(\scal{L_t}{\bar{\xi}_t}> u\big)\,\D u.
$$
With the help of~\eqref{thm_gks07}, we can conclude for fixed $u>0$ and $\varepsilon>0$ that
\begin{align*}
\erwxi{U(t)}
&\ge \varepsilon t\E^{(u-\varepsilon)t}\,\pp_0^{\ssup{t}}\times\mathrm{Prob}\big(\scal{L_t}{\bar{\xi}_t}> u \big)\\
&= \exp\Big(t\big[u-\varepsilon -(\kappa^*)^{d/(d+2)}\chi^{[\mathrm{GKS07}]}_H(u)\big](1+o(1))\Big)
\end{align*}
as $t\to\infty$. Now let $\varepsilon\downarrow 0$, take the supremum over all~$u>0$ and use~\eqref{gks07_rwrs} for $\beta=(\kappa^*)^{-d/(d+2)}$ to finish the proof of the lower bound.

For the upper bound, we can first derive an analogue formula to~\eqref{compact_upperbound} to restrict the support of the local times on a~compact box (see \cite[Prop.~4.4.3]{S10} for details). Therefore, it suffices to consider $U_{R\alpha_t}(t)=\ee_{0,R}^{\ssup{t}}[\exp(t\scaltext{L_t}{\bar{\xi}_t})]$ for some large $R>0$ instead of~$U(t)$. We will use a~similar strategy as in the proof of the upper bound in~\cite[Thm.~1.3]{GKS07}: In order to be able to apply the LDP for the local times, we need to smooth the scenery, which we can only do after cutting it. For~$M>0$, introduce $\bar{\xi}_t^{\ssup{\le M}}=(\bar{\xi}_t \wedge M)\vee (-M)$ and $\bar{\xi}_t^{\ssup{>M}}=(\bar{\xi}_t-M)_+$. Then $\bar{\xi}_t\le \bar{\xi}_t^{\ssup{\le M}}+\bar{\xi}_t^{\ssup{>M}}$. We want to work with the convolution $\bar{\xi}_t^{\ssup{\le M}}\star j_\delta$ with $j_\delta=\delta^{-d}j(\cdot/\delta)$, where $j\ge 0$ is a~smooth, rotational invariant, $\mathrm{L}^1$-normalised function supported in~$Q_1$. For brevity, we will not explain in detail how to deal with the remainder terms $\ee_{0,R}^{\ssup{t}}[\exp(t\scaltext{L_t}{\bar{\xi}_t^{\ssup{>M}}})]$ and $\ee_{0,R}^{\ssup{t}}[\exp(t\scaltext{L_t}{\bar{\xi}_t^{\ssup{\le M}}-\bar{\xi}_t^{\ssup{\le M}}\star j_\delta})]$ (which can be separated from the main term by H\"older's inequality). For the smoothing, one can apply \cite[Lemma~3.5]{GKS07}, while the cutting is technically involved and follows the proof of \cite[(2.12)]{GK09} (here we need Lemma~\ref{regvar_h}(b) and $\gamma<2$). Let us in the following take for granted that it is enough to show
\begin{multline}
\label{asymp_phase4_enough}
\limsup_{M\to\infty}\limsup_{\delta\downarrow 0}\limsup_{t\to\infty}
    \frac{1}{t\kappa^*}\log\erwxi{\ee_{0,R}^{\ssup{t}}\big[\exp\big(t\scal{L_t}{\bar{\xi}_t^{\ssup{\le M}}\star j_\delta}\big)\big]} \\[-2mm]
            \le -\chi^{\mathrm{(RWRS)}}_H\Big(\frac{1}{\kappa^*}\Big).
\end{multline}
Denote $\ell_t^{\ssup{\delta}}(z)=\int_{z+[0,1)^d}L_t\star j_\delta(y/\alpha_t) \,\D y$, then by rotational invariance of~$j$, we have $t\scaltext{L_t}{\bar{\xi}_t^{\ssup{\le M}}\star j_\delta}=\sum_{z\in\zd}\ell_t^{\ssup{\delta}}(z)\xi^{\ssup{\le M}}(z)$, and $\xi^{\ssup{\le M}}(z)=(\xi \wedge M)\vee (-M)\le \xi(z)\vee(-M)$. Furthermore,
\begin{align*}
\erwxi{\ee_{0,R}^{\ssup{t}}\big[\E^{\sum_{z\in\zd}\ell_t^{\ssup{\delta}}(z)\xi(z)\vee(-M)}\big]}
&\le \erwxi{\ee_{0,R}^{\ssup{t}}\big[\E^{\sum_{z\in\zd}\ell_t^{\ssup{\delta}}(z)\xi(z)}\big]\1_{\{\xi(z)>-M\}}}\\
&\quad\ + \erwxi{\ee_{0,R}^{\ssup{t}}\big[\E^{-M\sum_{z\in\zd}\ell_t^{\ssup{\delta}}(z)}\big]\1_{\{\xi(z)\le -M\}}}.
\end{align*}
The second summand is negligible on exponential scale~$t$ for $t\to\infty$ and $M\to\infty$ because of $\sum_{z\in\zd}\ell_t^{\ssup{\delta}}(z)=t$. In the first summand, the definition of~$H$ and Jensen's inequality (for the probability measure $\1_{\{z+[0,1)^d\}}\,\D y$) yield
$$
\Erwxi{\exp\Big(\sum_{z\in\zd}\ell_t^{\ssup{\delta}}(z)\xi(z)\Big)}
\le \exp\Big( t\int_{\rd}H(L_t\star j_\delta(y)) \,\D y\Big).
$$
Now we are ready to apply Varadhan's lemma to derive for any $M>0$
\begin{multline*}
\limsup_{t\to\infty}\frac{1}{t\kappa^*}\log\erwxi{\ee_{0,R}^{\ssup{t}}\big[\E^{t\int_{\rd}H(L_t\star j_\delta(y)) \,\D y}\big]\1_{\{\xi(z)>-M\}}}\\
\le -\inf_{\substack{g\in \mathrm{H}^1(\rd)\\ \supp g\subset Q_R\\ \norm{g}_2=1}}\Big\{%
    \int_{\rd}|\nabla g|^2-\frac{1}{\kappa^*}\int_{\rd}H\big(g^2\star j_\delta(y)\big)\,\D y\Big\}.
\end{multline*}
Again Jensen's inequality for the probability measure $j_\delta$ and Fubini's theorem show that we receive an upper bound when omitting the convolution with~$j_\delta$. Thus, we have arrived at a~compactified version of our variational problem~$\chi^{\mathrm{(RWRS)}}_H(1/\kappa^*)$, which we can estimate against the whole-space problem. This shows~\eqref{asymp_phase4_enough} and completes the proof of the theorem.


\begin{thebibliography}{[HKM06]}

\bibitem[BGT87]{BGT87}
{ N.~H.~Bingham, C.~M.~Goldie} and { J.~L.~Teugels}, \textit{Regular Variation},
                Cambridge University Press, Cambridge, 1987.



\bibitem[BK01]{BK01}
{ M.~Biskup} and { W.~K\"onig},
Long-time tails in the parabolic Anderson model with bounded potential,
                \textit{Ann.~Probab.} \textbf{29}(2), 636--682 (2001).




\bibitem[BK10]{BK10}
{ M.~Becker} and { W.~K\"onig},
                 Self-intersection local times of random walks: exponential moments in subcritical dimensions,
                 submitted (2010).



\bibitem[B07]{B07} { D.~Braess},
\textit{Finite elements. Theory, fast solvers and applications in elasticity theory}
        (Finite Elemente. Theorie, schnelle L\"oser und Anwendungen in der Elastizit\"atstheorie),
        4th revised and extended ed. (German), Springer, Berlin 2007.



\bibitem[BHK07]{BHK07}
{ D.~Brydges, R.~van der Hofstad} and { W.~K\"onig},
Joint density for the local times of continuous-time Markov chains,
                \textit{Ann. Probab.} \textbf{35}(4), 1307--1332 (2007).



\bibitem[CM94]{CM94}
{ R.A.~Carmona} and { S.A.~Molchanov},
\textit{Parabolic Anderson problem and intermittency},
                \textit{Mem. Am. Math. Soc.} \textbf{518}, 125 (1994).



\bibitem[CGH01]{CGH01}
{ Y.~Coudi\`ere, T.~Gallou\'et} and { R.~Herbin},
                Discrete Sobolev inequalities and $L^p$ error estimates for finite volume solutions of convection diffusion equations,
                \textit{M2AN, Math. Model. Numer. Anal.} \textbf{35}(4), 767--778 (2001).



\bibitem[DZ98]{DZ98}
{ A.~Dembo} and { O.~Zeitouni},
\textit{Large Deviations Techniques and Applications}, 2nd ed., Springer, New York, 1998.



\bibitem[DV75]{DV75}
{ M.D.~Donsker} and {S.R.S.~Varadhan},
Asymptotics for the Wiener sausage,
                \textit{Commun. Pure Appl. Math.} \textbf{28}, 525--565 (1975).



\bibitem[GKS07]{GKS07}
{ N.~Gantert, W.~K\"onig} and { Z.~Shi},
Annealed deviations of random walk in random scenery,
                \textit{Ann. Inst. Henri Poincar\'e, Probab. Stat.} \textbf{43}(1), 47--76 (2007).



\bibitem[GH99]{GH99}
{ J.~G\"artner} and { F.~den Hollander},
Correlation structure of intermittency in the parabolic Anderson model,
                \textit{Probab. Theory Relat. Fields} \textbf{114}(1), 1--54 (1999).



\bibitem[GK05]{GK05}
{ J.~G\"artner} and { W.~K\"onig},
The parabolic Anderson model, in: J.-D.~Deu\-schel (Ed.) et al., \textit{Interacting Stochastic Systems}, 153--179, Springer, Berlin 2005.



\bibitem[GM90]{GM90}
{ J.~G\"artner} and { S.A.~Molchanov},
Parabolic problems for the Anderson model.
                        I: Intermittency and related topics,
                \textit{Commun. Math. Phys.} \textbf{132}(3), 613--655 (1990).



\bibitem[GM98]{GM98}
{ J.~G\"artner} and { S.A.~Molchanov},
Parabolic problems for the Anderson model.
                        II: Second-order asymptotics and structure of high peaks,
                \textit{Probab. Theory Relat. Fields} \textbf{111}(1), 17--55 (1998).



\bibitem[GK09]{GK09}
{ G.~Gr\"uninger} and { W.~K\"onig},
Potential confinement property of the parabolic Anderson model,
                \textit{Ann. Inst. Henri Poincar\'e, Probab. Stat.} \textbf{45}(3), 840--863 (2009).



\bibitem[HKM06]{HKM06}
{ R.~van der Hofstad, W.~K\"onig} and { P.~M\"orters},
The universality classes in the parabolic Anderson model,
                \textit{Commun. Math. Phys.} \textbf{267}(2), 307--353 (2006).



\bibitem[LL01]{LL01}
{ E.H.~Lieb} and { M.~Loss},
\textit{Analysis}, 2nd ed.,
                American Mathematical Society, Providence, RI 2001.



\bibitem[S09]{S09}
{ B.~Schmidt},
On a semilinear variational problem, to appear in: \textit{ESAIM Control Optim. Calc. Var.}



\bibitem[S10]{S10}
{ S.~Schmidt}, Das parabolische Anderson-Modell mit Be- und Entschleunigung (German), PhD thesis, University of Leipzig (2010).


\end{thebibliography}
\end{document}